\documentclass[11pt]{amsart}

\usepackage[T1]{fontenc}
\usepackage{lmodern}
\usepackage{microtype}
\usepackage{mathtools}
\usepackage{amssymb}
\usepackage{hyperref}

\usepackage{xcolor}

\definecolor{linkblue}{RGB}{38,68,105}
\definecolor{citegreen}{RGB}{45,88,68}

\hypersetup{colorlinks=true,linkcolor=linkblue,citecolor=citegreen,urlcolor=linkblue,pdftitle={Derived Reconstruction of Semisimple Tensor and Module Categories},pdfauthor={Sheng Tan}}

\newcommand{\A}{\mathcal A}
\newcommand{\C}{\mathcal C}
\newcommand{\D}{\mathcal D}
\newcommand{\M}{\mathcal M}
\newcommand{\N}{\mathcal N}
\newcommand{\T}{\mathcal T}
\newcommand{\Hh}{\mathcal H}
\newcommand{\one}{\mathbf 1}
\newcommand{\Irr}{\operatorname{Irr}}
\newcommand{\add}{\operatorname{add}}
\newcommand{\dev}{\operatorname{dev}}
\newcommand{\MT}{\operatorname{MT}}
\newcommand{\Vect}{\operatorname{Vec}}
\newcommand{\Mod}{\operatorname{Mod}}

\DeclareMathOperator{\id}{id}
\DeclareMathOperator{\Hom}{Hom}
\DeclareMathOperator{\Ext}{Ext}

\newtheorem{theorem}{Theorem}[section]
\newtheorem{proposition}[theorem]{Proposition}
\newtheorem{lemma}[theorem]{Lemma}
\newtheorem{corollary}[theorem]{Corollary}
\theoremstyle{definition}
\newtheorem{definition}[theorem]{Definition}
\newtheorem{example}[theorem]{Example}
\theoremstyle{remark}

\theoremstyle{plain}

\newtheorem{theoremA}{Theorem}

\newtheorem{theoremB}{Theorem}

\newtheorem{theoremC}{Theorem}

\title[Derived Reconstruction]{Derived Reconstruction of Semisimple Tensor and Module Categories}
\author{Sheng Tan}
\address{School of Mathematical Sciences, Capital Normal University, Beijing 100048, China} 
\email{tansheng2018@outlook.com}
\date{September 6, 2026.}

\begin{document}

\begin{abstract}
We study bounded derived categories of semisimple tensor categories and their semisimple module categories. These categories may have infinitely many simple isomorphism classes. We classify bounded monoidal $t$-structures by integer-valued characters of the universal grading group and an integer deviation. We also classify compatible module $t$-structures using stabilizer subgroups of the module components. The corresponding hearts are tensor and module equivalent to the original categories. It follows that a monoidal derived equivalence and a coherent derived module equivalence recover both the tensor category and its module category.
\end{abstract}

\maketitle

\section{Introduction}

Let $k$ be an algebraically closed field. We consider essentially small locally finite semisimple tensor categories over $k$. We follow the conventions of \cite{EGNO}. In particular, the tensor unit is simple and every object has finite length. We also consider nonzero essentially small locally finite semisimple $k$-linear left module categories. Their actions are $k$-bilinear and exact in each variable. The sets of simple isomorphism classes may be infinite. All functors are understood to be $k$-linear. A monoidal functor means a strong monoidal functor. We write $\odot$ for module actions. For a monoidal functor $F:\C\to\D$, we write the structure isomorphisms of an $F$-module functor $G:\M\to\N$ as $G(X\odot M)\xrightarrow{\sim}F(X)\odot G(M)$.

An ordinary triangulated equivalence between bounded derived categories need not determine the underlying abelian categories. A monoidal triangulated equivalence contains more information. It is therefore natural to ask whether this additional structure determines the original tensor category. A bounded $t$-structure recovers an abelian category as its heart. However, a triangulated category may admit many bounded $t$-structures. Thus the reconstruction problem requires us to understand those $t$-structures that are compatible with the tensor product.

Throughout this paper, we use the definitions of monoidal $t$-structures and equivalence of bounded $t$-structures given by Xu and Zheng \cite[Definitions~2.6 and 2.10]{XuZheng}. Their definition of a monoidal $t$-structure was motivated by the work of Zhang and Zhou \cite{ZhangZhou}. In the normalized case, it agrees with Biglari's definition of a compatible bounded $t$-structure \cite{Biglari}.

The reconstruction results in \cite{XuZheng} combine uniform comparison of monoidal $t$-structures with tensor-reducedness. When there are finitely many simple isomorphism classes, the shift functions of bounded $t$-structures differ by bounded functions. Yu obtained related reconstruction results in this case \cite[Theorem~4.5 and Corollary~4.7]{Yu}. He also raised the corresponding question for tensor categories with infinitely many simple isomorphism classes in the paragraph following \cite[Corollary~4.7]{Yu}. Theorem~\ref{theoremA} answers this question in the semisimple case.

The case of infinitely many simple isomorphism classes is different. Let $\C=\Vect_{\mathbb Z}^{\mathrm{fd}}$, and let $L_n$ be the one-dimensional object in degree $n$. For each $c\in\mathbb Z$, the rule $L_n[a]\mapsto L_n[a+cn]$ defines a monoidal triangulated autoequivalence of $D^b(\C)$. For $c\ne0$, the shift functions of the transported and standard $t$-structures differ by the unbounded function $n\mapsto cn$. Thus these two $t$-structures are not equivalent. Its heart is still tensor equivalent to $\C$.

Our first main result classifies bounded monoidal $t$-structures on $D^b(\C)$ in terms of the universal grading group $U_{\C}$. We write $\MT(D^b(\C))$ for the set of bounded monoidal $t$-structures on $D^b(\C)$. We denote its subset of normalized structures by $\MT_0(D^b(\C))$.

\begin{theoremA}\label{theoremA}
Let $\C$ be an essentially small locally finite semisimple tensor category. There are bijections 
\begin{equation*}
    \MT(D^b(\C)) \longleftrightarrow \Hom(U_{\C},\mathbb Z)\times\mathbb Z, \quad \MT_0(D^b(\C)) \longleftrightarrow \Hom(U_{\C},\mathbb Z).
\end{equation*}
The pair $(\chi,d)$ corresponds to the shift function $p(X)=\chi(\deg X)+d$. The corresponding $t$-structure has deviation $\{d\}$. Moreover, if $\D$ is another essentially small locally finite semisimple tensor category, then
\[
D^b(\C)\simeq_\otimes D^b(\D)\quad\Longrightarrow\quad\C\simeq_\otimes\D.
\]
\end{theoremA}

For a character $\chi: U_{\C}\to\mathbb Z$, the universal grading gives a decomposition $\C=\bigoplus_n\C_n$. The normalized heart associated with $\chi$ is tensor equivalent to $\C$ through
\begin{equation*}
R_\chi(X)=\bigoplus_n X_n[n].
\end{equation*}
This gives tensor reconstruction without uniform comparison of $t$-structures. Moreover, let $F: D^b(\C)\to D^b(\D)$ be a monoidal triangulated equivalence. There is a unique $\chi_F\in\Hom(U_{\D},\mathbb Z)$ such that $T_{-\chi_F}\circ F$ is $t$-exact for the standard $t$-structures.

We next consider the pair $(D^b(\C),D^b(\M))$ with its derived module action. Since $D^b(\M)$ has no intrinsic tensor product, the preceding tensor reconstruction argument cannot be applied directly. We therefore study $t$-structures on $D^b(\M)$ that are compatible with the monoidal $t$-structure on $D^b(\C)$. We use the definitions of compatibility and triangulated module $t$-structures in \cite[Definitions~3.1 and 3.4]{Yu} throughout the paper.

The results of \cite[Theorem~4.5 and Corollary~4.7]{Yu} concern abelian subcategories with finitely many simple isomorphism classes. Their bounded derived categories are left triangulated tensor ideals. The reconstruction equivalence is obtained by restricting an ambient monoidal functor. Here we consider semisimple module categories with a given coherent derived module equivalence.

Stroi\'nski and Stroi\'nski--Zorman study reconstruction of module categories from cyclic generators, internal algebraic or Tambara data, coalgebra objects, or module monads \cite{Stroinski,StroinskiZorman}. Group-set gradings and stabilizer subgroups also appear in Clifford theory for strongly graded tensor categories and for fusion categories graded by finite groups \cite{Galindo,GalindoGraded,MeirMusicantov}.

Fix a connected component $\Omega$ of the graph of simple objects under the module action and choose $M_0\in\Omega$. Let $H_\Omega\subseteq U_{\C}$ consist of the degrees of homogeneous objects whose action on $M_0$ contains $M_0$ as a direct summand. Then $H_\Omega$ is a subgroup of $U_{\C}$, and $\Omega$ is graded by the left $U_{\C}$-set $U_{\C}/H_\Omega$. Our second main result classifies the compatible module $t$-structures.

\begin{theoremB}\label{theoremB}
Let $\M$ be nonzero. Let the monoidal $t$-structure on $D^b(\C)$ have shift function $q(X)=\chi(\deg X)+d$, and let $p\colon\Irr(\M)\to\mathbb Z$ determine a bounded $t$-structure on $D^b(\M)$. An integer $m$ is a module deviation if and only if
\[
p(P)=q(X)+p(M)-m
\]
whenever $X\in\Irr(\C)$, $M,P\in\Irr(\M)$, and $P$ is a direct summand of $X\odot M$. 
A module deviation exists if and only if
\[
p(P)=p(M)+\chi(\deg X)
\]
for every such triple. In this case, $\dev_{t_q}(t_p)={d}$.
For a component $\Omega$, a function $p$ satisfying this relation exists if and only if $\chi(H_\Omega)=0$. When this condition holds, all such functions on $\Omega$ have the form
\[
p(P)=c_\Omega+\chi(g),\quad P\in\M_{gH_\Omega},
\]
where $c_\Omega\in\mathbb Z$ is arbitrary. The constants on different components can be chosen independently.
\end{theoremB}

The constants in Theorem~\ref{theoremB} are independent for distinct components. Thus the standard $t$-structure on $D^b(\C)$ allows one arbitrary integer shift on each module component. For a nonzero character, the stabilizer condition determines whether the regrading of $D^b(\C)$ extends to the module category.

Suppose that $p$ satisfies the condition in Theorem~\ref{theoremB}. Let $\M_r$ be generated by the simple objects with $p(M)=r$. Then $\C_n\odot\M_r\subseteq\M_{n+r}$, and
\[
R_p(M)=\bigoplus_r M_r[r]
\]
is an $R_\chi$-module equivalence from $\M$ to the corresponding heart. We use the same shift isomorphism and Koszul sign for the tensor and module structures on these functors. These equivalences give the following reconstruction theorem.

\begin{theoremC}\label{theoremC}
Let $\C$ and $\D$ be essentially small locally finite semisimple tensor categories, and let $\M$ and $\N$ be nonzero essentially small locally finite semisimple left module categories over $\C$ and $\D$, respectively. Suppose that $F: D^b(\C)\to D^b(\D)$ is a monoidal triangulated equivalence and that $G: D^b(\M)\to D^b(\N)$ is a triangulated equivalence with a coherent $F$-module-functor structure. Then there are an exact $k$-linear tensor equivalence $f:\C\to\D$ and an exact $k$-linear $f$-module equivalence $g:\M\to\N$.
\end{theoremC}

Theorem~\ref{theoremC} also applies to decomposable module categories. Each of the four categories may have infinitely many simple isomorphism classes. Finite length still gives finite decompositions into simple shifts and homogeneous components. Example~\ref{ex:nonsemisimple} gives nonsemisimple module categories over $\Vect^{\mathrm{fd}}$ with equivalent derived categories but inequivalent abelian categories. For separable algebra objects, Theorem~\ref{theoremC} recovers the associated module category. It therefore recovers the Morita class of the algebra after transport by $f$.

The organization of the paper is as follows. Section~\ref{sec:monoidal-t-str} classifies bounded and monoidal $t$-structures on semisimple derived categories. Section~\ref{sec:compatible-t-str} treats compatible module $t$-structures, components, and stabilizers. Section~\ref{sec:reconstruction} constructs the tensor and module regrading functors and proves the reconstruction theorems. Section~\ref{sec:examples} gives examples and an application to separable algebra objects.

\subsubsection*{Acknowledgment}

The author is supported by NSFC (Grant No.~12601116), BJNSF (Grant No.~1264052), and a Beijing municipal research program for returned overseas scholars. He thanks Zhian Jia for many helpful discussions.

\section{Monoidal \texorpdfstring{$t$}{t}-structures on semisimple derived categories}\label{sec:monoidal-t-str}

\subsection{Bounded \texorpdfstring{$t$}{t}-structures}\label{subsec:bounded-t-structures}

We use cochain complexes and the shift convention
\begin{equation*}
(K[a])^i=K^{i+a},\quad d_{K[a]}^i=(-1)^a d_K^{i+a}.
\end{equation*}
Thus if $S$ is regarded as a stalk complex concentrated in degree $0$, then $S[a]$ is concentrated in degree $-a$. We use the standard definition of a $t$-structure \cite[Section~1.3]{BBD}. We write it as $(D^{\leq0},D^{\geq1})$ and put $D^{\geq0}=D^{\geq1}[1]$. We also write $D^{\leq n}=D^{\leq0}[-n]$ and $D^{\geq n}=D^{\geq0}[-n]$. For a collection $\mathcal S$ of objects in an additive category, let $\add(\mathcal S)$ be the full subcategory of direct summands of finite direct sums of objects of $\mathcal S$. For $X\in\C$, we write $X^*$ for a chosen left dual, with evaluation $X^*\otimes X\to\one$.

\begin{lemma}\label{lem:split-derived}
Let $\A$ be an essentially small semisimple abelian category in which every object has finite length. Then every object of $D^b(\A)$ is a finite direct sum of objects $S[a]$ with $S\in\Irr(\A)$ and $a\in\mathbb Z$. Moreover,
\[
\Hom_{D^b(\A)}(S[a],T[b])\cong\Ext_{\A}^{b-a}(S,T),
\]
so this group is zero unless $a=b$ and $S\cong T$.
\end{lemma}

\begin{proof}
Since $\A$ is semisimple, every short exact sequence in $\A$ splits. Hence every bounded complex is isomorphic in the derived category to the direct sum of its cohomology objects placed in the corresponding degrees. Only finitely many cohomology objects are nonzero, and each has finite length. Therefore every object of $D^b(\A)$ is a finite direct sum of shifts $S[a]$ with $S\in\Irr(\A)$. For simple objects $S,T\in\A$,
$$
\Hom_{D^b(\A)}(S[a],T[b])\cong \Ext_{\A}^{b-a}(S,T).
$$
Since $\A$ is semisimple, all positive extension groups vanish. Also, $\Hom_{\A}(S,T)=0$ unless $S\cong T$. This proves the last statement.
\end{proof}

\begin{proposition}\label{prop:semisimple-t}
For every bounded $t$-structure on $D^b(\A)$, there is a unique function $p:\Irr(\A)\to\mathbb Z$ such that $S[p(S)]$ belongs to its heart for every simple object $S$. Conversely, every such function defines a bounded $t$-structure $t_p$ by
\[
D_p^{\leq0}=\add\{S[a]:a\geq p(S)\},\quad D_p^{\geq1}=\add\{S[a]:a<p(S)\}.
\]
Its coaisle and heart are
\[
D_p^{\geq0}=\add\{S[a]:a\leq p(S)\},\quad \Hh_p=\add\{S[p(S)]:S\in\Irr(\A)\}.
\]
\end{proposition}

\begin{proof}
Let $(\mathcal U,\mathcal V)$ be a bounded $t$-structure, with $\mathcal U=D^{\leq0}$ and $\mathcal V=D^{\geq1}$. Fix a simple shift $S[a]$ and consider its truncation triangle
\[
U\rightarrow S[a]\rightarrow V\rightarrow U[1],\quad U\in\mathcal U,\quad V\in\mathcal V.
\]
Suppose that the first map is nonzero. By Lemma~\ref{lem:split-derived}, its restriction to some summand of $U$ isomorphic to $S[a]$ is nonzero. This restriction is an isomorphism, so the first map is a split epimorphism. Hence $S[a]\in\mathcal U$. If the first map is zero, then $V\cong S[a]\oplus U[1]$. Thus $S[a]\in\mathcal V$. Orthogonality prevents $S[a]$ from belonging to both $\mathcal U$ and $\mathcal V$. These two subcategories therefore partition the simple shifts.

For fixed $S$, the set $\{a:S[a]\in\mathcal U\}$ is upward closed because $\mathcal U[1]\subseteq\mathcal U$. Boundedness makes this set nonempty and proper. It is therefore equal to $\{a:a\geq p(S)\}$ for a unique integer $p(S)$. The descriptions of $\mathcal V$, the coaisle, and the heart follow.

Conversely, define the two subcategories by the displayed formulas. Their shift closure and orthogonality follow from Lemma~\ref{lem:split-derived}. For any object, separate the summands in its decomposition into those with $a\geq p(S)$ and those with $a<p(S)$. The associated split triangle is a truncation triangle. Each object has only finitely many summands, so the $t$-structure is bounded. The function $p$ itself need not be bounded.
\end{proof}

By Proposition~\ref{prop:semisimple-t},
\begin{equation}\label{eq:shifted-halves}
S[a]\in D_p^{\leq m}\Longleftrightarrow a\geq p(S)-m,\quad S[a]\in D_p^{\geq m}\Longleftrightarrow a\leq p(S)-m.
\end{equation}

\subsection{Monoidal \texorpdfstring{$t$}{t}-structures}\label{subsec:monoidal-classification}

We use the notion of Xu and Zheng \cite[Definition~2.10]{XuZheng}.

\begin{definition}\label{def:monoidal-t-structure}
Let $(\T,\otimes,\one)$ be a monoidal triangulated category. Let $t=(\T^{\leq0},\T^{\geq1})$ be a bounded $t$-structure on $\T$. The $t$-structure $t$ is \emph{monoidal} if there exists $d\in\mathbb Z$ such that
\begin{equation}\label{eq:monoidal-t}
\T^{\leq0}\otimes\T^{\leq d}\subseteq\T^{\leq0},\quad \T^{\geq0}\otimes\T^{\geq d}\subseteq\T^{\geq0}.
\end{equation}
Its \emph{deviation} $\dev(t)$ is the set of all integers $d$ satisfying \eqref{eq:monoidal-t}. A monoidal $t$-structure is \emph{normalized} if $0\in\dev(t)$.
\end{definition}

Every monoidal $t$-structure becomes normalized after a shift, and a normalized monoidal $t$-structure has deviation $\{0\}$ \cite[Lemma~2.11 and Proposition~2.15]{XuZheng}.

For simple objects $X,Y,Z$ of $\C$, write $Z\leq X\otimes Y$ when $Z$ is a direct summand of $X\otimes Y$. The universal grading group $U_{\C}$ has a degree map $\deg:\Irr(\C)\to U_{\C}$, which satisfies $\deg Z=\deg X\deg Y$ whenever $Z\leq X\otimes Y$. Let $G$ be a group and let $u:\Irr(\C)\to G$ satisfy the same relation. Then there is a unique homomorphism $\overline u: U_{\C}\to G$ such that $u(X)=\overline u(\deg X)$. This universal property is proved for semisimple tensor categories in \cite[Lemma~2.1]{ENOtwisted}; see also \cite[Section~4.14]{EGNO}.

\begin{theorem}\label{thm:monoidal-classification}
There is a bijection 
\begin{equation*}
    \MT(D^b(\C)) \longleftrightarrow \Hom(U_{\C},\mathbb Z)\times\mathbb Z.
\end{equation*}
More precisely, suppose that $t_p$ is monoidal and $d\in\dev(t_p)$. Then there is a unique homomorphism $\chi: U_{\C}\to\mathbb Z$ such that
\begin{equation}\label{eq:p-character-d}
p(X)=\chi(\deg X)+d
\end{equation}
for every simple object $X$. Conversely, every pair $(\chi,d)$ defines a monoidal $t$-structure by this formula with deviation $\{d\}$.
\end{theorem}

\begin{proof}
Suppose that $d\in\dev(t_p)$, and let $Z\leq X\otimes Y$ for simple objects $X,Y,Z$. The objects $X[p(X)]$ and $Y[p(Y)-d]$ lie in $D_p^{\leq0}$ and $D_p^{\leq d}$, respectively. The first inclusion in \eqref{eq:monoidal-t} places the summand $Z[p(X)+p(Y)-d]$ in $D_p^{\leq0}$. Thus $p(X)+p(Y)-d\geq p(Z)$. The same objects also lie in $D_p^{\geq0}$ and $D_p^{\geq d}$, respectively. The second inclusion gives the reverse inequality. Hence
\begin{equation}\label{eq:fusion-p}
p(Z)=p(X)+p(Y)-d.
\end{equation}
Taking $X=Y=Z=\one$ gives $p(\one)=d$. Thus $X\mapsto p(X)-d$ is additive on simple summands of tensor products. The universal property of $U_{\C}$ gives a unique $\chi: U_{\C}\to\mathbb Z$ satisfying \eqref{eq:p-character-d}.

Conversely, fix $(\chi,d)$ and define $p$ by \eqref{eq:p-character-d}. If $X[a]\in D_p^{\leq0}$ and $Y[b]\in D_p^{\leq d}$, then $a\geq p(X)$ and $b\geq p(Y)-d$. Every simple summand $Z$ of $X\otimes Y$ satisfies $p(Z)=p(X)+p(Y)-d$. Hence $a+b\geq p(Z)$ and $Z[a+b]\in D_p^{\leq0}$. Taking finite direct sums gives the first inclusion in \eqref{eq:monoidal-t}. For $X[a]\in D_p^{\geq0}$ and $Y[b]\in D_p^{\geq d}$, the reverse inequalities give the second inclusion. Thus $d\in\dev(t_p)$. If $e$ is another deviation, the first part of the proof gives $e=p(\one)=d$.
\end{proof}

\begin{corollary}\label{cor:normalized-classification}
There is a bijection
\begin{equation*}
    \MT_0(D^b(\C)) \longleftrightarrow \Hom(U_{\C},\mathbb Z). 
\end{equation*}
The structure associated with $\chi$ has shift function $p(X)=\chi(\deg X)$ and heart
\[
\Hh_\chi=\add\{X[\chi(\deg X)]:X\in\Irr(\C)\}.
\]
\end{corollary}

\begin{proof}
Normalization is equivalent to $d=0$ in Theorem~\ref{thm:monoidal-classification}. Hence the shift function is $p(X)=\chi(\deg X)$. The description of the heart then follows from Proposition~\ref{prop:semisimple-t}.
\end{proof}

\subsection{Uniform comparison of \texorpdfstring{$t$}{t}-structures}\label{subsec:uniform-comparison}

Two bounded $t$-structures $t$ and $t'$ are \emph{equivalent} if there are integers $m\leq n$ such that
\[
D_t^{\leq m}\subseteq D_{t'}^{\leq0}\subseteq D_t^{\leq n}.
\]

\begin{proposition}\label{prop:uniform-comparison}
Let $p,q:\Irr(\C)\to\mathbb Z$. The bounded $t$-structures $t_p$ and $t_q$ are equivalent if and only if $p-q$ is bounded on $\Irr(\C)$.
\end{proposition}

\begin{proof}
By \eqref{eq:shifted-halves}, the inclusions $D_p^{\leq m}\subseteq D_q^{\leq0}\subseteq D_p^{\leq n}$ hold precisely when $m\leq p(S)-q(S)\leq n$ for every $S\in\Irr(\C)$.
\end{proof}

\begin{corollary}\label{cor:characters-not-comparable}
Distinct characters $\chi,\psi\in\Hom(U_{\C},\mathbb Z)$ define inequivalent normalized monoidal $t$-structures.
\end{corollary}

\begin{proof}
Let $\delta=\chi-\psi$ and choose $g\in U_{\C}$ with $\delta(g)\ne0$. Then $g$ has infinite order. The universal grading is faithful. Thus, for each $r\in\mathbb Z$, there is a simple object $S_r$ of degree $g^r$. The function $(p_\chi-p_\psi)(S_r)=r\delta(g)$ is unbounded. The result follows from Proposition~\ref{prop:uniform-comparison}.
\end{proof}

When $\Irr(\C)$ is finite, any two shift functions differ by a bounded function. This is the finiteness step in \cite[Lemma~4.3]{XuZheng} and \cite[Lemma~4.1]{Yu}. The reconstruction argument in \cite{XuZheng} then uses the uniqueness theorem for equivalent tensor reduced monoidal $t$-structures \cite[Theorem~2.23]{XuZheng}. The next lemma shows that $D^b(\C)$ is tensor reduced.

\begin{lemma}\label{lem:tensor-reduced}
If $0\ne K\in D^b(\C)$, then $K\otimes K\ne0$.
\end{lemma}

\begin{proof}
First let $A,Y$ be nonzero objects of $\C$. For fixed $Y$, the objects $B$ with $B\otimes Y=0$ form a left Serre tensor ideal. If this ideal contained a nonzero object $A$, it would also contain $A^*\otimes A$. The evaluation $A^*\otimes A\to\one$ is nonzero by the triangle identity. It is an epimorphism because $\one$ is simple. The ideal would then contain $\one$, which contradicts $Y\ne0$. Thus $A\otimes Y\ne0$.

Let $q$ be the largest integer with $H^q(K)\ne0$. Exactness of tensor product and semisimplicity give
\[
H^{2q}(K\otimes K)\cong H^q(K)\otimes H^q(K)\ne0.
\]
Thus $K\otimes K\ne0$.
\end{proof}

Thus normalized monoidal $t$-structures on a tensor reduced category can be inequivalent. Theorem~\ref{thm:tensor-regrading} below identifies their hearts directly.

\section{Compatible \texorpdfstring{$t$}{t}-structures and module components}\label{sec:compatible-t-str}


The exact action $\C\times\M\to\M$ induces a triangulated action on bounded derived categories. We use the following total complex.

\begin{lemma}\label{lem:derived-action}
For bounded complexes $K\in C^b(\C)$ and $L\in C^b(\M)$, set
\[
(K\odot L)^r=\bigoplus_{i+j=r}K^i\odot L^j,
\]
with differential
\begin{equation}\label{eq:total-action-differential}
d(x\odot y)=d_Kx\odot y+(-1)^i x\odot d_Ly,\quad x\in K^i.
\end{equation}
This construction induces a triangulated action $D^b(\C)\times D^b(\M)\to D^b(\M)$. The module associator acts degreewise, with no additional sign.
\end{lemma}

\begin{proof}
All sums are finite, so the total complex is well defined. Since every bounded acyclic complex in a semisimple abelian category is contractible, it is enough to check that the action preserves contractible complexes. 
If $h$ contracts $K$, define a homotopy $\tilde h$ on $K\odot L$ by $\tilde{h}(x\odot y)=h(x)\odot y$. Then $\tilde{h}$ contracts $K\odot L$. Similarly, if $h$ contracts $L$, then $\tilde{h}(x\odot y)=(-1)^i x\odot h(y)$ contracts $K\odot L$. Therefore the total action preserves quasi-isomorphisms in each variable and descends to the derived categories.

It remains to check the module structure. For $K_1,K_2\in C^b(\C)$ and $L\in C^b(\M)$, apply the module associator degreewise. On $K_1^i\odot K_2^j\odot L^r$, the three terms of the differential have coefficients $1$, $(-1)^i$, and $(-1)^{i+j}$ for either bracketing. Hence the associator is a chain map. Finally, the pentagon and unit diagrams reduce to the original diagrams in each total degree.
\end{proof}

\subsection{Compatible \texorpdfstring{$t$}{t}-structures}\label{subsec:compatible-t-structures}

Let $t_q$ be a monoidal $t$-structure on $D^b(\C)$ and let $t_p$ be a bounded $t$-structure on $D^b(\M)$. If $t_q$ is normalized, then $\Hh_q$ is tensor closed \cite[Proposition~2.13]{XuZheng}, so Yu's compatibility notion applies \cite[Definition~3.1]{Yu}.

\begin{definition}\label{def:compatible-module-product}
Suppose that $t_q$ is normalized. The module product bifunctor is \emph{compatible with $t_q$ and $t_p$} if
\[
\Hh_q\odot\Hh_p\subseteq\Hh_p.
\]
\end{definition}

For module deviation we use Yu's relative notion \cite[Definition~3.4]{Yu}.

\begin{definition}\label{def:triangulated-module-t}
The bounded $t$-structure $t_p$ is a \emph{triangulated module $t$-structure with respect to $t_q$} if there exists $m\in\mathbb Z$ such that
\begin{equation}\label{eq:module-deviation}
D_q^{\leq0}\odot D_p^{\leq m}\subseteq D_p^{\leq0},\quad D_q^{\geq0}\odot D_p^{\geq m}\subseteq D_p^{\geq0}.
\end{equation}
The set of all such integers is the \emph{deviation of $t_p$ with respect to $t_q$}, denoted by $\dev_{t_q}(t_p)$.
\end{definition}

By Theorem~\ref{thm:monoidal-classification}, the shift function of $t_q$ has the form
\begin{equation}\label{eq:ambient-shift}
q(X)=\chi(\deg X)+d.
\end{equation}

\begin{theorem}\label{thm:module-compatibility}
Assume that $\M$ is nonzero. Let $p:\Irr(\M)\to\mathbb Z$ determine a bounded $t$-structure on $D^b(\M)$. An integer $m$ belongs to $\dev_{t_q}(t_p)$ if and only if
\begin{equation}\label{eq:compatibility-general}
p(P)=q(X)+p(M)-m
\end{equation}
whenever $X\in\Irr(\C)$, $M,P\in\Irr(\M)$, and $P$ is a direct summand of $X\odot M$. Such an integer exists if and only if
\begin{equation}\label{eq:compatibility}
p(P)=p(M)+\chi(\deg X)
\end{equation}
for every such triple. In this case, $\dev_{t_q}(t_p)=\{d\}$. When $d=0$, relation \eqref{eq:compatibility} is also equivalent to $\Hh_\chi\odot\Hh_p\subseteq\Hh_p$.
\end{theorem}

\begin{proof}
Suppose first that $m$ satisfies \eqref{eq:module-deviation}. We have $X[q(X)]\in D_q^{\leq0}\cap D_q^{\geq0}$ and $M[p(M)-m]\in D_p^{\leq m}\cap D_p^{\geq m}$. Hence
\[
X[q(X)]\odot M[p(M)-m]\cong(X\odot M)[q(X)+p(M)-m]
\]
lies in $\Hh_p$. The heart is closed under direct summands, so $P[q(X)+p(M)-m]\in\Hh_p$. The unique shift of $P$ in the heart is $P[p(P)]$. This proves \eqref{eq:compatibility-general}.

Taking $X=\one$ and $P=M$ gives $m=q(\one)=d$. Substitution gives \eqref{eq:compatibility}.

Conversely, assume \eqref{eq:compatibility}. It is enough to check \eqref{eq:module-deviation} on simple shifts. If $X[a]\in D_q^{\leq0}$ and $M[b]\in D_p^{\leq d}$, then $a\geq q(X)$ and $b\geq p(M)-d$. For every simple summand $P$ of $X\odot M$,
\[
a+b\geq q(X)+p(M)-d=p(M)+\chi(\deg X)=p(P),
\]
so $P[a+b]\in D_p^{\leq0}$. The reverse inequalities give the coaisle inclusion. Thus $d$ is a module deviation. The calculation with $X=\one$ shows that it is the only one.

When $d=0$, the first part shows that \eqref{eq:module-deviation} is equivalent to \eqref{eq:compatibility}. Also, $X[\chi(\deg X)]\odot M[p(M)]$ lies in $\Hh_p$ exactly when each simple summand occurs with the shift prescribed by $p$. This is again \eqref{eq:compatibility}.
\end{proof}

\subsection{Module components and stabilizers}\label{subsec:module-components}

\begin{definition}\label{def:module-components}
The graph associated with $\M$ has vertex set $\Irr(\M)$. Two vertices $M$ and $P$ are joined when $P$ is a direct summand of $X\odot M$ for some $X\in\Irr(\C)$. A connected component of this graph is called a \emph{component of $\M$}.
\end{definition}

The adjacency condition is symmetric. Indeed, suppose that $P$ is a direct summand of $X\odot M$. By rigidity, a split inclusion $P\to X\odot M$ gives a nonzero map $X^*\odot P\to M$. Since $M$ is simple, this map is an epimorphism. Semisimplicity then implies that it splits. Hence $M$ is a direct summand of $X^*\odot P$.

For a component $\Omega$, let $\M_\Omega$ be the full additive subcategory generated by its vertices. Every simple constituent of $X\odot M$ lies in the same component as $M$. Hence each $\M_\Omega$ is stable under the $\C$-action, and
\[
\M=\bigoplus_\Omega\M_\Omega,
\]
where each object has only finitely many nonzero components. These are precisely the indecomposable semisimple module summands of $\M$. Any decomposition of $\M$ into stable semisimple subcategories gives a partition of the vertices of the graph. Since each subcategory is stable under the $\C$-action, every edge lies within one part of the partition. Thus every part is a union of connected components. Conversely, any union of connected components generates a stable subcategory. Therefore the subcategories $\M_\Omega$ are precisely the indecomposable semisimple module summands.

Fix a component $\Omega$ and $M_0\in\Omega$. For $g\in U_{\C}$, let $\C_g$ denote the homogeneous component of the universal grading.

\begin{definition}\label{def:component-stabilizer}
Define $H_\Omega\subseteq U_{\C}$ by
\begin{equation}\label{eq:stabilizer}
H_\Omega=\{g\in U_{\C}:M_0\text{ is a direct summand of }X_g\odot M_0\text{ for some }X_g\in\C_g\}.
\end{equation}
\end{definition}

\begin{proposition}\label{prop:coset-grading}
The subset $H_\Omega$ is a subgroup of $U_{\C}$. Every $P\in\Omega$ is a direct summand of $X_g\odot M_0$ for some $g\in U_{\C}$ and some $X_g\in\C_g$. If $P$ is also a direct summand of $Y_h\odot M_0$ for some $Y_h\in\C_h$, then $gH_\Omega=hH_\Omega$.
\end{proposition}

\begin{proof}
The identity belongs to $H_\Omega$ by the unit constraint. Now let $g,h\in H_\Omega$. Choose $X_g\in\C_g$ and $Y_h\in\C_h$ such that $M_0$ is a direct summand of both $X_g\odot M_0$ and $Y_h\odot M_0$. Composing the corresponding split inclusions gives
$$
M_0\rightarrow X_g\odot M_0\rightarrow X_g\odot(Y_h\odot M_0)\cong (X_g\otimes Y_h)\odot M_0.
$$
Hence $gh\in H_\Omega$. Next let $g\in H_\Omega$. Choose a split inclusion $M_0\to X_g\odot M_0$. By rigidity, this gives a nonzero map $X_g^*\odot M_0\rightarrow M_0$. Since $M_0$ is simple, this map is an epimorphism. Semisimplicity implies that it splits. Since $X_g^*\in\C_{g^{-1}}$, we obtain $g^{-1}\in H_\Omega$. Thus $H_\Omega$ is a subgroup of $U_{\C}$.

Now let $P\in\Omega$. Choose a path from $M_0$ to $P$. Using the symmetry of adjacency, each step can be written in the direction of the path by replacing the acting object with its dual when necessary. Composing the split inclusions along the path shows that $P$ is a direct summand of $X_g\odot M_0$ for some homogeneous object $X_g\in\C_g$.

Suppose that $P$ is also a direct summand of $Y_h\odot M_0$ for some $Y_h\in\C_h$. A split inclusion $P\to X_g\odot M_0$ gives, by rigidity, a nonzero map $X_g^*\odot P\rightarrow M_0$. Since $M_0$ is simple, this map is a split epimorphism. Choose a section $M_0\rightarrow X_g^*\odot P$. Composing it with the map induced by a split inclusion $P\to Y_h\odot M_0$ gives a nonzero map
$$
M_0\rightarrow (X_g^*\otimes Y_h)\odot M_0.
$$
This map splits because $\M$ is semisimple. Since $X_g^*\otimes Y_h\in\C_{g^{-1}h}$, it follows that $g^{-1}h\in H_\Omega$. Therefore $gH_\Omega=hH_\Omega$.
\end{proof}

We call $H_\Omega$ the stabilizer subgroup associated with $M_0$. For $P\in\Omega$, define $\deg_\Omega(P)=gH_\Omega$ when $P$ is a direct summand of $X_g\odot M_0$. This is well defined by Proposition~\ref{prop:coset-grading}. For each $gH_\Omega\in U_{\C}/H_\Omega$, let $\M_{gH_\Omega}$ be the full additive subcategory generated by the simple objects of degree $gH_\Omega$. Then
\[
\M_\Omega=\bigoplus_{gH_\Omega\in U_{\C}/H_\Omega}\M_{gH_\Omega},
\]
where each object has only finitely many nonzero homogeneous components. This grading is compatible with the $\C$-action. Indeed, if $X\in\C_x$ and $M\in\M_{gH_\Omega}$, then every simple summand of $X\odot M$ has degree $xgH_\Omega$. Hence
\begin{equation}\label{eq:coset-action}
\C_x\odot\M_{gH_\Omega}\subseteq\M_{xgH_\Omega}.
\end{equation}

If we choose another base simple object of degree $gH_\Omega$, the same argument gives the stabilizer $gH_\Omega g^{-1}$. Since $\mathbb Z$ is abelian, every homomorphism $\chi: U_{\C}\to\mathbb Z$ is invariant under conjugation. Therefore the condition in the next theorem does not depend on the choice of $M_0$.

\begin{theorem}\label{thm:stabilizer}
Fix $\chi: U_{\C}\to\mathbb Z$. On a component $\Omega$, a function $p$ satisfying \eqref{eq:compatibility} exists if and only if
\begin{equation}\label{eq:stabilizer-condition}
\chi(H_\Omega)=0.
\end{equation}
When this condition holds, all such functions have the form
\begin{equation}\label{eq:component-functions}
p(P)=c_\Omega+\chi(g),\quad P\in\M_{gH_\Omega},
\end{equation}
where $c_\Omega\in\mathbb Z$ is arbitrary. The constants on distinct components are independent.
\end{theorem}

\begin{proof}
Let $h\in H_\Omega$. By definition, there exists $X_h\in\C_h$ such that $M_0$ is a direct summand of $X_h\odot M_0$. Decompose $X_h$ into simple objects. Since $M_0$ is simple, it is a direct summand of $Y_h\odot M_0$ for some simple constituent $Y_h$ of $X_h$. Since $X_h\in\C_h$, we also have $Y_h\in\C_h$. Applying \eqref{eq:compatibility} gives $p(M_0)=p(M_0)+\chi(h)$. Thus $\chi(h)=0$. Since $h$ was arbitrary, \eqref{eq:stabilizer-condition} is necessary.

Now let $P$ have degree $gH_\Omega$. By definition of the degree, there exists $X_g\in\C_g$ such that $P$ is a direct summand of $X_g\odot M_0$. As above, we may choose a simple constituent $Y_g\in\C_g$ such that $P$ is a direct summand of $Y_g\odot M_0$. Applying \eqref{eq:compatibility} gives $p(P)=p(M_0)+\chi(g)$. Hence $p$ has the form \eqref{eq:component-functions}, with $c_\Omega=p(M_0)$.

Conversely, suppose that \eqref{eq:stabilizer-condition} holds and define $p$ by \eqref{eq:component-functions}. This is well defined. Indeed, if $gH_\Omega=g'H_\Omega$, then $g^{-1}g'\in H_\Omega$, so $\chi(g')-\chi(g)=\chi(g^{-1}g')=0$. Now let $P\in\M_{gH_\Omega}$ and let $X\in\C_x$ be simple. If $Q$ is a simple direct summand of $X\odot P$, then \eqref{eq:coset-action} gives $Q\in\M_{xgH_\Omega}$. Therefore
\[
p(Q)=c_\Omega+\chi(xg)=p(P)+\chi(x),
\]
which is \eqref{eq:compatibility}. Finally, the constant $c_\Omega$ can be chosen independently on each component.
\end{proof}

For the standard $t$-structure on $D^b(\C)$, we have $\chi=0$. Hence Theorem~\ref{thm:stabilizer} gives one arbitrary integer shift on each component. If there are infinitely many components, these shifts may be unbounded. Nevertheless, the resulting $t$-structure is bounded because every object of $D^b(\M)$ is a finite direct sum of shifts of simple objects. The $U_{\C}/H_\Omega$-grading above is related to the group-set gradings studied in \cite{Galindo,GalindoGraded,MeirMusicantov}.

\section{Regrading and reconstruction}\label{sec:reconstruction}

We first give the shift isomorphism for tensor products and module actions. Let $\star$ denote either operation. We use the total-complex differential from \eqref{eq:total-action-differential}.

\begin{lemma}\label{lem:shift-isomorphism}
Let $K\in C^b(\C)$ and let $L$ be a bounded complex in $\C$ or $\M$. For integers $a,b$, there is a natural chain isomorphism
\[
\sigma_{a,b}: K[a]\star L[b]\rightarrow(K\star L)[a+b]
\]
defined by
\begin{equation}\label{eq:shift-isomorphism}
\sigma_{a,b}(x\star y)=(-1)^{bi}x\star y,\quad x\in K[a]^i.
\end{equation}
For three factors, these isomorphisms are compatible with associativity.
\end{lemma}

\begin{proof}
The differential on the source is
\[
d(x\star y)=(-1)^a d_Kx\star y+(-1)^{i+b}x\star d_Ly.
\]
After applying \eqref{eq:shift-isomorphism}, the coefficients of the two terms are $(-1)^{a+b(i+1)}=(-1)^{bi+a+b}$ and $(-1)^{bi+i+b}$. On the other hand, applying $\sigma_{a,b}$ first and then the differential on the target gives the same coefficients. Hence $\sigma_{a,b}$ is a chain isomorphism. It is natural in $K$ and $L$, so it induces a natural isomorphism on the bounded derived categories.

It remains to check compatibility with associativity. Let $x\in K[a]^i$ and $y\in L[b]^j$, and let the third factor be shifted by $c$. The two composites in the associativity diagram have signs
\[
(-1)^{bi+c(i+j)}\quad\text{and}\quad(-1)^{cj+(b+c)i}.
\]
The two signs are equal. After cancelling the common scalar, the diagram reduces to the original associativity diagram. Since the total-complex associator acts degreewise, it contributes no additional sign.
\end{proof}

\subsection{Tensor regrading and reconstruction}\label{subsec:tensor-reconstruction}

Fix $\chi: U_{\C}\to\mathbb Z$ and write
\[
\C=\bigoplus_{n\in\mathbb Z}\C_n,
\]
where $\C_n$ is generated by the simple objects $X$ with $\chi(\deg X)=n$. There are no morphisms between distinct homogeneous components. Hence every object has a functorial decomposition
\[
X=\bigoplus_n X_n,\quad X_n\in\C_n.
\]
Since every object has finite length, only finitely many $X_n$ are nonzero. The corresponding projection functors $X\mapsto X_n$ are exact.

The normalized aisle and coaisle associated with $\chi$ are closed under tensor product. Thus $\Hh_\chi$ is a tensor subcategory of $D^b(\C)$ with the inherited tensor product.

\begin{theorem}\label{thm:tensor-regrading}
The functor
\[
R_\chi:\C\rightarrow\Hh_\chi,\quad R_\chi(X)=\bigoplus_n X_n[n],
\]
admits a strong monoidal structure and is an exact tensor equivalence.
\end{theorem}

\begin{proof}
Semisimplicity gives
\[
\Hom_{\Hh_\chi}(R_\chi(X),R_\chi(Y))\cong\bigoplus_n\Hom_{\C}(X_n,Y_n)\cong\Hom_{\C}(X,Y).
\]
Every object of $\Hh_\chi$ is a finite direct sum of simple objects with their prescribed shifts. Hence $R_\chi$ is an equivalence. It is exact because both categories are semisimple.

We next define the monoidal structure on $R_\chi$. Let the tensorator
\[
\mu_{X,Y}: R_\chi(X)\otimes R_\chi(Y)\rightarrow R_\chi(X\otimes Y)
\]
be defined on the summand $X_n[n]\otimes Y_m[m]$ by composing $\sigma_{n,m}$ from Lemma~\ref{lem:shift-isomorphism} with the canonical inclusion
\[
(X_n\otimes Y_m)[n+m]\rightarrow(X\otimes Y)_{n+m}[n+m].
\]
Since tensor product distributes over the finite homogeneous decompositions, the sum of these maps is an isomorphism. Naturality follows from Lemma~\ref{lem:shift-isomorphism} and the functoriality of the homogeneous projections.

It remains to check coherence. The stalk complex $X_n[n]$ is concentrated in degree $-n$. Thus the scalar in \eqref{eq:shift-isomorphism} is $(-1)^{nm}$. For three homogeneous objects of degrees $n,m,\ell$, the signs along the two sides of the monoidal associativity diagram are
\[
(-1)^{nm+(n+m)\ell} \quad\text{and}\quad (-1)^{m\ell+n(m+\ell)}.
\]
These signs agree. After cancelling the common scalar, the diagram reduces to the associativity diagram in $\C$.

Finally, since $\one\in\C_0$, we take the identity $\one\to R_\chi(\one)$ as the unit map. The corresponding signs are all $1$. Hence the unit diagrams reduce to those in $\C$. Therefore $R_\chi$ is strong monoidal.
\end{proof}

\begin{proposition}\label{prop:derived-regrading}
The functor
\[
T_\chi: D^b(\C)\rightarrow D^b(\C),\quad T_\chi(K)=\bigoplus_n K_n[n],
\]
is a strong monoidal triangulated autoequivalence. Its restriction to the standard heart is $R_\chi$. In particular, it sends the standard heart onto $\Hh_\chi$.
\end{proposition}

\begin{proof}
Apply the exact homogeneous projections termwise to a bounded complex. This gives a functorial decomposition $K=\bigoplus_n K_n$ into homogeneous subcomplexes. On $D^b(\C_n)$, the functor $T_\chi$ is the shift $[n]$. Hence $T_\chi$ is a triangulated autoequivalence, with inverse given by $[-n]$ on each $D^b(\C_n)$.

The tensorator is defined on homogeneous summands by the isomorphisms $\sigma_{n,m}$ from Lemma~\ref{lem:shift-isomorphism}. Their compatibility with associativity gives the monoidal coherence. Since $\one\in\C_0$, the unit constraint is the identity. Thus $T_\chi$ is strong monoidal.

Finally, if $X\in\C$ is regarded as an object of the standard heart, then
\[
T_\chi(X)=\bigoplus_n X_n[n]=R_\chi(X).
\]
Hence the restriction of $T_\chi$ to the standard heart is $R_\chi$, and its image is $\Hh_\chi$.
\end{proof}

\begin{theorem}\label{thm:tensor-reconstruction}
Let $\C$ and $\D$ be essentially small locally finite semisimple tensor categories. If $F: D^b(\C)\to D^b(\D)$ is a monoidal triangulated equivalence, then there is an exact $k$-linear tensor equivalence $\C\simeq_\otimes\D$.
\end{theorem}

\begin{proof}
Transport the standard $t$-structure on $D^b(\C)$ along $F$. The resulting $t$-structure on $D^b(\D)$ is bounded and normalized monoidal. By Corollary~\ref{cor:normalized-classification}, its heart is $\Hh_\chi$ for a unique character $\chi: U_{\D}\to\mathbb Z$. By the definition of the transported $t$-structure, the restriction of $F$ gives an exact tensor equivalence $\C\rightarrow\Hh_\chi$. On the other hand, Theorem~\ref{thm:tensor-regrading} gives an exact tensor equivalence $R_\chi:\D\rightarrow\Hh_\chi$. Composing the first equivalence with a tensor quasi-inverse of $R_\chi$ gives an exact tensor equivalence $\C\to\D$.
\end{proof}

\begin{corollary}\label{cor:t-exact-regrading}
Let $F: D^b(\C)\to D^b(\D)$ be a monoidal triangulated equivalence. There is a unique $\chi_F\in\Hom(U_{\D},\mathbb Z)$ such that $T_{-\chi_F}\circ F$ is $t$-exact for the standard $t$-structures. If $\Hom(U_{\D},\mathbb Z)=0$, then every monoidal triangulated equivalence $D^b(\C)\to D^b(\D)$ is $t$-exact for the standard structures.
\end{corollary}

\begin{proof}
Let $\chi_F$ be the character corresponding to the transported standard $t$-structure. Its shift function is $X\mapsto \chi_F(\deg X)$. By Proposition~\ref{prop:derived-regrading}, the autoequivalence $T_{-\chi_F}$ sends this $t$-structure to the standard one. Hence $T_{-\chi_F}\circ F$ is $t$-exact. It remains to prove uniqueness. Suppose that $T_{-\eta}\circ F$ is $t$-exact for another character $\eta$. Then $\chi_F(\deg X)-\eta(\deg X)=0$ for every simple object $X$. Therefore $\chi_F$ and $\eta$ agree on all degrees of simple objects. By the universal property of $U_{\D}$, we have $\eta=\chi_F$.
\end{proof}

\subsection{Module regrading}\label{subsec:module-regrading}

Fix $\chi: U_{\C}\to\mathbb Z$ and a function $p:\Irr(\M)\to\mathbb Z$ satisfying \eqref{eq:compatibility}. Let $\C=\bigoplus_n\C_n$ be the grading induced by $\chi$, and define
\[
\M_r=\add\{M\in\Irr(\M):p(M)=r\}.
\]
Then $\M=\bigoplus_r\M_r$. Since every object has finite length, every object $M\in\M$ has a finite functorial decomposition $M=\bigoplus_r M_r$ with $M_r\in\M_r$. The corresponding projection functors are exact. The compatibility equation \eqref{eq:compatibility} gives 
\begin{equation}\label{eq:graded-action}
\C_n\odot\M_r\subseteq\M_{n+r}.
\end{equation}

\begin{theorem}\label{thm:module-regrading}
The functor
\[
R_p:\M\rightarrow\Hh_p,\quad R_p(M)=\bigoplus_r M_r[r],
\]
is an exact equivalence. It admits a coherent $R_\chi$-module-functor structure and is an $R_\chi$-module equivalence.
\end{theorem}

\begin{proof}
Since the homogeneous projections are functorial, $R_p$ is additive. Semisimplicity gives
\[
\Hom_{\Hh_p}(R_p(M),R_p(N))\cong\bigoplus_r\Hom_{\M}(M_r,N_r)\cong\Hom_{\M}(M,N).
\]
Thus $R_p$ is fully faithful. Every object of $\Hh_p$ is a finite direct sum of simple objects with their prescribed shifts, so $R_p$ is essentially surjective. Hence $R_p$ is an equivalence. It is exact because both categories are semisimple.

We next define the module structure. Let the tensorator 
\[
\rho_{X,M}: R_\chi(X)\odot R_p(M)\rightarrow R_p(X\odot M)
\]
be defined on the summand $X_n[n]\odot M_r[r]$ by composing $\sigma_{n,r}$ from Lemma~\ref{lem:shift-isomorphism} with the inclusion
\[
(X_n\odot M_r)[n+r]\rightarrow(X\odot M)_{n+r}[n+r].
\]
By \eqref{eq:graded-action}, each summand lands in the correct homogeneous component. Since the homogeneous decompositions are finite, these maps together give an isomorphism. Naturality follows from the naturality of $\sigma$ and the functoriality of the projections.

The stalk complex $X_n[n]$ is concentrated in degree $-n$, so the scalar in $\sigma_{n,r}$ is $(-1)^{nr}$. It remains to check coherence. Let $X\in\C_n$, $Y\in\C_s$, and $M\in\M_r$. In the module associativity diagram, the two sides give the signs $(-1)^{ns+(n+s)r}$ and $(-1)^{sr+n(s+r)}$. These signs agree. After cancelling the common scalar, the diagram reduces to the module associativity diagram in $\M$. The total-complex associator contributes no additional sign.

Finally, since $\one\in\C_0$, the sign in the unit diagram is $1$, so this diagram reduces to the unit diagram in $\M$. Thus $\rho$ is coherent. With the convention fixed in this paper, the module-functor constraint of $R_p$ is $\rho_{X,M}^{-1}$.
\end{proof}

\subsection{Reconstruction of module categories}\label{subsec:module-reconstruction}

\begin{proposition}\label{prop:transported-pair}
Let $F: D^b(\C)\to D^b(\D)$ be a monoidal triangulated equivalence, and let $G: D^b(\M)\to D^b(\N)$ be a triangulated equivalence with a coherent $F$-module-functor structure. Transport the standard $t$-structures along $F$ and $G$. The resulting $t$-structure on $D^b(\D)$ is associated with a unique character $\chi: U_{\D}\to\mathbb Z$. Let $p:\Irr(\N)\to\mathbb Z$ be the shift function of the resulting $t$-structure on $D^b(\N)$. Then
\[
p(P)=p(M)+\chi(\deg X)
\]
whenever $X\in\Irr(\D)$, $M,P\in\Irr(\N)$, and $P$ is a direct summand of $X\odot M$.
\end{proposition}

\begin{proof}
The $t$-structure on $D^b(\D)$ is normalized monoidal. Hence Corollary~\ref{cor:normalized-classification} gives a unique character $\chi$. For the standard $t$-structures, the derived action satisfies
\[
D_{\C}^{\leq0}\odot D_{\M}^{\leq0}\subseteq D_{\M}^{\leq0},\quad D_{\C}^{\geq0}\odot D_{\M}^{\geq0}\subseteq D_{\M}^{\geq0}.
\]
Using the monoidal structure of $F$ and the module-functor structure of $G$, these inclusions pass to the transported $t$-structures on $D^b(\D)$ and $D^b(\N)$. Thus $0$ is a module deviation for this pair. The formula now follows from Theorem~\ref{thm:module-compatibility}.
\end{proof}

\begin{theorem}\label{thm:simultaneous-reconstruction}
Let $\C$ and $\D$ be essentially small locally finite semisimple tensor categories, and let $\M$ and $\N$ be nonzero essentially small locally finite semisimple left module categories over $\C$ and $\D$, respectively. Suppose that $F: D^b(\C)\to D^b(\D)$ is a monoidal triangulated equivalence and that $G: D^b(\M)\to D^b(\N)$ is a triangulated equivalence with a coherent $F$-module-functor structure. Then there are an exact $k$-linear tensor equivalence $f:\C\to\D$ and an exact $k$-linear $f$-module equivalence $g:\M\to\N$.
\end{theorem}

\begin{proof}
Let $\Hh_\chi$ and $\Hh_p$ be the hearts of the transported $t$-structures from Proposition~\ref{prop:transported-pair}. Restriction gives exact equivalences
\begin{equation*}
\overline F=F|_{\C}:\C\rightarrow\Hh_\chi,\quad \overline G=G|_{\M}:\M\rightarrow\Hh_p.
\end{equation*}
The tensorator of $F$ makes $\overline F$ strong monoidal, and the module constraint of $G$ makes $\overline G$ an $\overline F$-module equivalence. Let
\begin{equation*}
R=R_\chi:\D\rightarrow\Hh_\chi,\quad P=R_p:\N\rightarrow\Hh_p
\end{equation*}
be the equivalences constructed above. Write $\mu_{X,Y}: R(X)\otimes R(Y)\to R(X\otimes Y)$ for the tensorator of $R$, and let $\rho_{X,M}: R(X)\odot P(M)\to P(X\odot M)$ be the isomorphism from Theorem~\ref{thm:module-regrading}.

Choose quasi-inverses $S:\Hh_\chi\to\D$ and $Q:\Hh_p\to\N$, with counits
\begin{equation*}
\varepsilon: RS\xrightarrow{\sim}\id_{\Hh_\chi},\quad
\delta: PQ\xrightarrow{\sim}\id_{\Hh_p}.
\end{equation*}
We first give $S$ a monoidal structure. Define the tensorator $S(A)\otimes S(B)\to S(A\otimes B)$ as the unique map whose image under $R$ satisfies
\begin{equation}\label{eq:S-tensorator}
\varepsilon_{A\otimes B}\circ R\bigl(S(A)\otimes S(B)\to S(A\otimes B)\bigr)\circ\mu_{S(A),S(B)}
=\varepsilon_A\otimes\varepsilon_B.
\end{equation}
Define the unit map in the same way. Since $R$ is fully faithful, these maps are uniquely determined. Naturality follows from that of $\mu$ and $\varepsilon$. The coherence of $\mu$ then gives the associativity and unit conditions. Thus $S$ is strong monoidal and $\varepsilon$ is monoidal.

We next give $Q$ an $S$-module structure. For $A\in\Hh_\chi$ and $B\in\Hh_p$, define
\begin{equation*}
\kappa_{A,B}: S(A)\odot Q(B)\rightarrow Q(A\odot B)
\end{equation*}
as the unique map satisfying
\begin{equation}\label{eq:kappa}
\delta_{A\odot B}\circ P(\kappa_{A,B})\circ\rho_{S(A),Q(B)}
=\varepsilon_A\odot\delta_B.
\end{equation}
Such a map exists and is an isomorphism because $P$ is fully faithful. Its naturality follows from that of $\rho$, $\varepsilon$, and $\delta$. Applying $P$ to the module associativity diagram and using \eqref{eq:kappa} reduces it to the associativity diagram for $\rho$ together with the monoidality of $\varepsilon$. The unit condition follows in the same way. Hence $\kappa^{-1}$ gives a coherent module-functor structure on $Q$.

Finally, invert the restricted module constraint of $\overline G$ and write it as
\begin{equation*}
\gamma_{X,M}:
\overline F(X)\odot\overline G(M)
\xrightarrow{\sim}
\overline G(X\odot M).
\end{equation*}
Set $f=S\overline F$ and $g=Q\overline G$. Then $f$ is a tensor equivalence. Define
\begin{equation*}
\beta_{X,M}
Q(\gamma_{X,M})\circ
\kappa_{\overline F(X),\overline G(M)}
:
f(X)\odot g(M)\rightarrow g(X\odot M).
\end{equation*}
The associativity condition for $\beta$ follows from those for $\kappa$ and $\gamma$, together with the monoidal structure on $f$. The unit condition follows in the same way. Thus $\beta^{-1}$ gives a coherent $f$-module-functor structure on $g$. Therefore $g$ is an $f$-module equivalence. Since $f$ and $g$ are equivalences between semisimple abelian categories, they are exact.
\end{proof}

The argument applies componentwise. Therefore no uniform bound on $p$ is needed across different components, and $p$ may be unbounded on $\Irr(\N)$.

\section{Examples}\label{sec:examples}

\begin{example}\label{ex:graded-vector-spaces}
Let $\C=\Vect_{\mathbb Z}^{\mathrm{fd}}$ be the tensor category of finite-dimensional $\mathbb Z$-graded vector spaces, and let $L_n$ be the one-dimensional object concentrated in degree $n$. Then $L_m\otimes L_n\cong L_{m+n}$ and $U_{\C}\cong\mathbb Z$. Every character $\mathbb Z\to\mathbb Z$ is multiplication by an integer $c$. The corresponding normalized heart is
\[
\Hh_c=\add\{L_n[cn]:n\in\mathbb Z\}.
\]
The tensorator of $R_c$ on $L_m\otimes L_n$ is multiplication by $(-1)^{c^2mn}$. Proposition~\ref{prop:derived-regrading} gives
\[
T_c(L_n[a])=L_n[a+cn].
\]
If $c\ne0$, this autoequivalence is not $t$-exact for the standard $t$-structure. The resulting $t$-structure is not equivalent to the standard one. Its heart is nevertheless tensor equivalent to $\C$.

Now let $\M=\Vect^{\mathrm{fd}}$, with the action induced by the forgetful tensor functor $\Vect_{\mathbb Z}^{\mathrm{fd}}\to\Vect^{\mathrm{fd}}$. Since $L_n\odot k\cong k$ for every $n$, the unique simple object of $\M$ has stabilizer $U_{\C}\cong\mathbb Z$. Hence Theorem~\ref{thm:stabilizer} shows that the character $n\mapsto cn$ extends to a compatible $t$-structure on $D^b(\M)$ only when $c=0$.
\end{example}

\begin{corollary}\label{cor:unique-standard}
The standard normalized monoidal $t$-structure on $D^b(\C)$ is unique if and only if $\Hom(U_{\C},\mathbb Z)=0$.
\end{corollary}

\begin{proof}
The standard structure corresponds to the zero homomorphism in Corollary~\ref{cor:normalized-classification}.
\end{proof}

If $\Irr(\C)$ is finite, then $U_{\C}$ is finite and $\Hom(U_{\C},\mathbb Z)=0$. For each $d\in\mathbb Z$, the constant shift function $p(S)=d$ defines a monoidal $t$-structure with deviation $\{d\}$. Thus the uniqueness statement in Corollary~\ref{cor:unique-standard} applies to normalized structures.

\begin{example}\label{ex:nonsemisimple}
Let
\[
A=k(1\rightarrow2\rightarrow3),\quad B=k(1\leftarrow2\rightarrow3)
\]
be path algebras. A Bernstein--Gelfand--Ponomarev reflection gives a $k$-linear triangulated equivalence \cite{Happel}
\[
D^b(A\text{-mod})\simeq D^b(B\text{-mod}).
\]
The abelian categories $A\text{-mod}$ and $B\text{-mod}$ are not equivalent. Indeed, the algebras are basic and their Gabriel quivers are not isomorphic. On the other hand, the reflection equivalence is given by derived tensor product with a bounded complex of bimodules. Hence it is linear over $D^b(\Vect^{\mathrm{fd}})$. Taking $\C=\D=\Vect^{\mathrm{fd}}$ and $F=\id$, we obtain a coherent derived module equivalence. Thus the semisimplicity assumption in Theorem~\ref{thm:simultaneous-reconstruction} cannot be omitted.
\end{example}

Two algebra objects in a tensor category are Morita equivalent when their module categories are equivalent as module categories \cite[Definition~7.8.17]{EGNO}.

\begin{corollary}\label{cor:algebra-objects}
Let $A\in\C$ and $B\in\D$ be separable algebra objects. Suppose that a pair $(F,G)$ as in Theorem~\ref{thm:simultaneous-reconstruction} relates
\[
D^b(\Mod_{\C}(A))\quad\text{and}\quad D^b(\Mod_{\D}(B)).
\]
Then the reconstructed tensor equivalence $f:\C\to\D$ carries $A$ to an algebra object Morita equivalent to $B$.
\end{corollary}

\begin{proof}
Let $A$ be a separable algebra object. Separability gives an $A$-linear section of the action map of every right $A$-module. Hence every right $A$-module is a direct summand of a free module.

Since $-\otimes A$ is exact, kernels and cokernels of $A$-module maps are computed in $\C$. Thus the forgetful functor $U\colon\Mod_{\C}(A)\to\C$ is exact. For every $X\in\C$,
\[
\Hom_A(X\otimes A,-)\cong\Hom_{\C}(X,U(-)).
\]
Every object of $\C$ is projective, so every free $A$-module is projective. Since every $A$-module is a direct summand of a free module, every $A$-module is projective. Therefore $\Mod_{\C}(A)$ is semisimple.

Moreover, $\Hom_A(M,N)$ is a subspace of $\Hom_{\C}(M,N)$, and every $A$-module has finite length as an object of $\C$. Hence $\Mod_{\C}(A)$ is locally finite. The same argument applies to $B$.

Theorem~\ref{thm:simultaneous-reconstruction} gives an $f$-module equivalence
\[
\Mod_{\C}(A)\simeq\Mod_{\D}(B).
\]
Applying $f$ to the $A$-module structures identifies $\Mod_{\C}(A)$ with $\Mod_{\D}(f(A))$. Hence $\Mod_{\D}(f(A))$ and $\Mod_{\D}(B)$ are equivalent as $\D$-module categories. Therefore $f(A)$ and $B$ are Morita equivalent.
\end{proof}

Morita equivalent algebra objects need not be isomorphic. For example, the separable algebras $k$ and $M_n(k)$ in $\Vect^{\mathrm{fd}}$ have equivalent module categories but are not isomorphic for $n>1$.

\end{document}